\documentclass[11pt,fleqn]{amsart}
\usepackage{amsmath}
\usepackage{amssymb}
\usepackage[all]{xy}           
\usepackage{amssymb}           
\usepackage{fullpage}
\usepackage{hyperref}
\usepackage{eucal}
\usepackage{graphicx}
\usepackage{float}
\usepackage{epsfig}
\usepackage[usenames,dvipsnames]{color}
\usepackage{charter}
\usepackage{color}
\usepackage{comment}
\usepackage{hyperref}
\usepackage{fancyvrb} 
\numberwithin{equation}{section}

\newtheorem{definition}{Definition}[section]
\newtheorem{lemma}[definition]{Lemma}
\newtheorem{theorem}[definition]{Theorem}
\newtheorem{proposition}[definition]{Proposition}

\newtheorem{remarkth}[definition]{Remark}

\newtheorem{question}[definition]{Question}

\begin{document}

\title[Poisson structures on wrinkled fibrations]{Poisson structures on wrinkled fibrations}

\author{P. Su\'arez-Serrato}

\address{Instituto de Matem\'aticas, Universidad Nacional Aut\'onoma de M\'exico\\Circuito Exterior, Ciudad Universitaria\\Coyoac\'an, 04510\\Mexico City\\Mexico.}

\email{pablo@im.unam.mx}
 \urladdr{http://www.matem.unam.mx/PabloSuarezSerrato}

\author{J. Torres Orozco}

\address{CIMAT\\ Jalisco S/N, Valenciana\\ Guanajuato\\ Mexico.}

\email{jonatan@cimat.mx}

\begin{abstract} We provide local formul{\ae} for Poisson bivectors and symplectic forms on the leaves of Poisson structures associated to wrinkled fibrations on smooth $4$--manifolds.
\end{abstract}

\subjclass[2010]{57R17, 53D17.}
\keywords{Poisson structures, smooth $4$-manifolds, wrinkled fibrations}

\maketitle

\section{Introduction}

The study of smooth manifolds of dimension four has led to various interesting types of fibrations. One of the origins of this research direction can be found in the seminal paper by Auroux, Donaldson, and Katzarkov \cite{ADK05}. They described near-symplectic forms adapted to a type of fibration that has since been called {\it broken Lefschetz fibration}, this name was introduced by Perutz \cite{P07-2}.

\begin{definition}
\label{D:BLF}
On a smooth $4$--manifold $X$, a {\it broken Lefschetz fibration} is a smooth map $f\colon X \rightarrow S^2$ that is a submersion outside the singularity set.  Moreover, the allowed singularities are of the following type:
\begin{enumerate}
\item {\it Lefschetz} singularities:  finitely many points
\begin{equation*}
\{
p_1, \dots , p_k\} \subset X,
\end{equation*}
which are locally modeled by
complex charts
\begin{equation*} {\bf C}^{2} \rightarrow {\bf C}  ,  \quad \quad (z_1, z_2) \mapsto z_{1}^{2} + z_{2}^{2},
\end{equation*}

\item {\it indefinite fold} singularities, also called {\it
broken}, contained in the smooth embedded 1-dimensional
submanifold $\Gamma \subset X \setminus  \{ p_1, \dots , p_k\} $,
which are locally modelled by the real charts
\begin{equation*} {\bf R}^{4} \rightarrow{\bf R}^{2} ,  \quad \quad  (t,x_1,x_2,x_3) \mapsto (t, - x_{1}^{2} + x_{2}^{2} + x_{3}^{2}).
\end{equation*}
\end{enumerate}
\end{definition}

The term indefinite in (ii) refers to the fact that the quadratic form $- x_{1}^{2} + x_{2}^{2} + x_{3}^{2}$ is neither negative nor positive definite. In the language of singularity theory, these subsets are known as fold singularities of corank 1. Since $X$ is closed, $\Gamma$ is homeomorphic to a collection of disjoint circles.
 For this reason, throughout this work, we will often refer to $\Gamma$ as {\it singular circles}.  On the other hand, we can only assert that $f(\Gamma)$ is  a union of immersed curves. In particular, the images of the components of $\Gamma$ need not be disjoint, and the image of each component can self-intersect.

In \cite{GSV15} the first named author together with Garc\'ia-Naranjo and Vera exhibited a Poisson structure whose symplectic leaves coincide with the fibres of a broken Lefschetz fibration, and the singular sets of both structures coincide.

The notion of {\it wrinkled fibration} on a smooth $4$-manifold was introduced by Lekili \cite{L09}, who showed that these wrinkled fibrations exist in every closed oriented smooth $4$--manifold. Broken Lefschetz fibrations are not stable, as maps. In contrast, wrinkled fibrations are stable. So if one is interested in perturbations of broken Lefschetz fibrations, one is led naturally to the study of wrinkled fibrations. Lekili described a set of moves that describe, up to homotopy, all the possible one-parameter deformations of broken and wrinkled fibrations. These were then further studied by Williams \cite{W10}. Both \cite{L09} and \cite{W10} are motivated by the {\it Lagrangian matching invariants} defined by Perutz \cite{P07-2}. These moves preserve the diffeomorphism type of the underlying smooth $4$--manifold.

Let $X$ be a closed $4$--manifold, and $\Sigma$ be a $2$--dimensional surface. A  map $f: X \to \Sigma$ is said to have a cusp singularity at a point $p$ in $X$, if around $p$, $f$ is locally modeled in oriented charts by the map:
\begin{equation*}(t, x, y, z) \mapsto  (t, x^3 -3xt + y^2 - z^2).
\end{equation*}
The critical point set is a smooth arc, $\{ x^2=t, y=0, z=0 \}$,
the critical value set is a cusp given by $\{ (t,s) : 4t^3=s^2
\}$ (see figure \ref{fig:cusp}).
\begin{figure}
    \centering
    \includegraphics[width=0.4\textwidth]{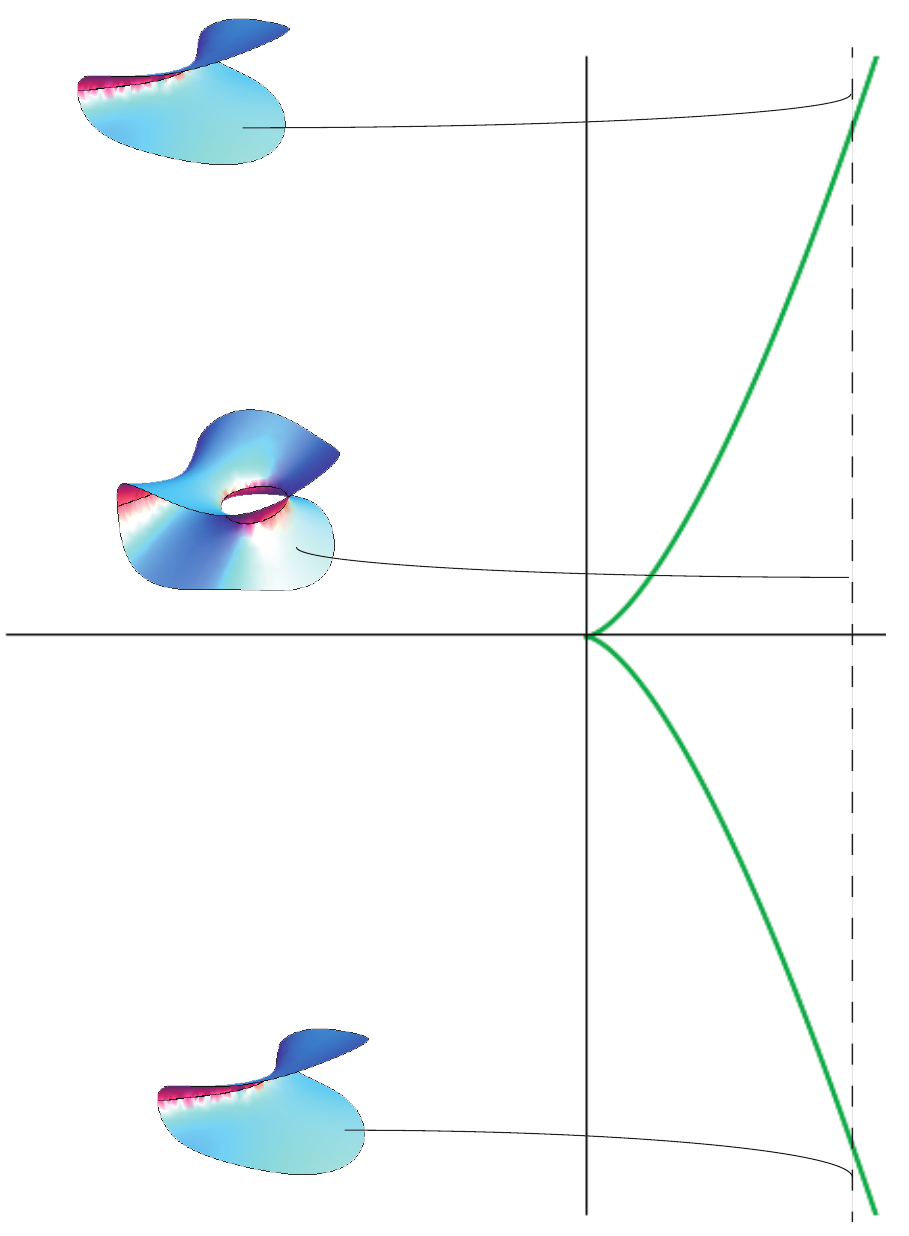}
    \caption{A diagram depicting a fibration with a cusp singularity along the critical values depicted in green, which is a subset of the base of this fibration. The black lines indicate the points over which each of the fibres lie.}
    \label{fig:cusp}
\end{figure}
Following Lekili we state:
 \begin{definition} A wrinkled fibration on a closed $4$--manifold $X$ is a smooth map $f$ to a closed surface which is a broken fibration when restricted to $X\setminus C$, where $C$ is a finite set such that around each point in $C$, $f$ has cusp singularities. We say that a fibration is {\it purely wrinkled} if it has no isolated Lefschetz-type singularities.
\end{definition}

Wrinkled fibrations may be obtained from broken Lefschetz fibrations by performing wrinkling moves. These eliminate a Lefschetz type singularity and introduce a wrinkled fibration structure. Conversely, it is possible to modify a wrinkled fibration locally by smoothing out the cusp singularity by introducing a Lefschetz type singularity, so obtaining a broken fibration (see \cite{L09, W10}).

Next we will introduce the deformations of wrinkled fibrations that will appear in our Poisson strucures. As Lekili, by a {\it deformation of wrinkled fibrations} we mean a one-parameter family of maps which is a wrinkled fibration for all but finitely many values. One of Lekili's major contributions in \cite{L09} was to show that any one-parameter family deformation of a purely wrinkled fibration is homotopic (relative endpoints) to one which realizes a sequence of births, merges, flips, their inverses, and isotopies staying within the class of purely wrinkled fibrations. Moreover, these moves do not change the diffeomorphism type of the $4$--manifold $X$ on which they take place.

Let us briefly describe these moves, readers may consult both \cite{L09, W10} for the corresponding descriptions in terms of how the fibres change and how these moves can be described using handlebodies. The moves we are interested in are to be considered as maps ${\bf R} \times {\bf R}^3 \to {\bf R}$, given by the following equations, each depending on a real parameter $s$:
\medskip
\\ {\bf Move 1 (Birth, fig. \ref{fig:birth})}
\begin{equation*}
b_s(x, y, z, t)=(t, x^3-3x(t^2-s)+y^2-z^2)
\end{equation*}

\begin{figure}
    \centering
    \includegraphics[width=1\textwidth]{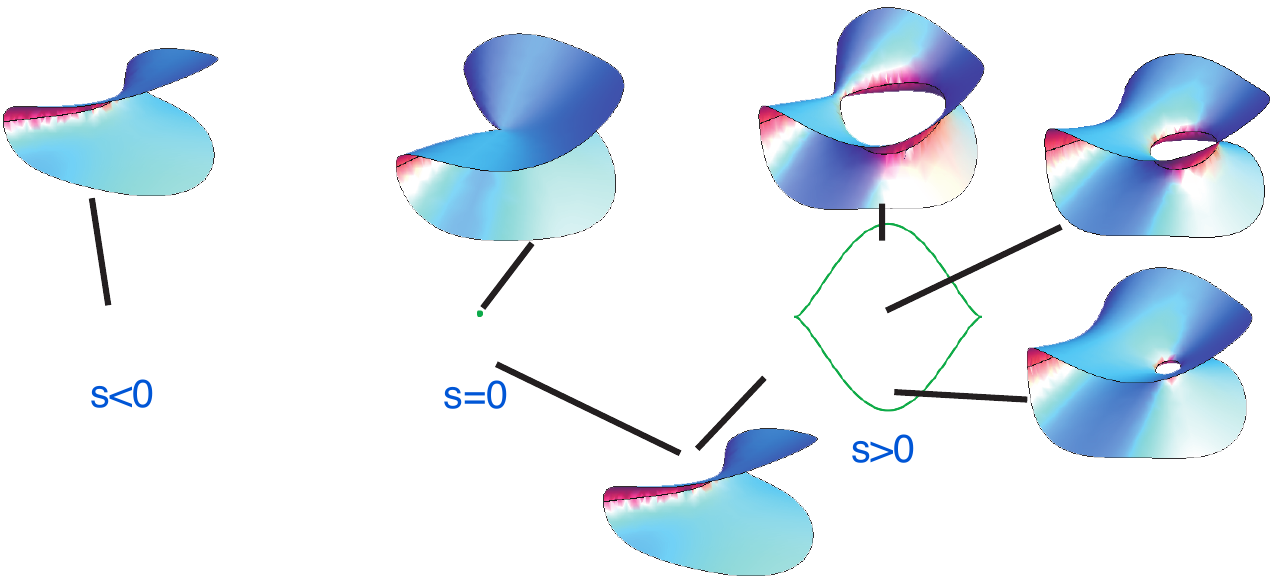}
    \caption{ {\bf Birth Move:} $b_s(x, y, z, t)=(t, x^3-3x(t^2-s)+y^2-z^2)$.  For $s<0$ the critical set is empty. Then when $s=0$ the fibre above the critical point is shown to develop a singularity. As $s$ becomes postive the wrinkled critical set appears, here depicted by the green line, which is a subset of the base of this fibration. The black lines indicate the points over which each of the fibres lies.}
    \label{fig:birth}
\end{figure}
\medskip
 \noindent{\bf Move 2 (Merging, fig. \ref{fig:merge})}
\begin{equation*}
m_s(x, y, z, t)=(t, x^3-3x(s-t^2)+y^2-z^2)
\end{equation*}
\begin{figure}
    \centering
    \includegraphics[width=1\textwidth]{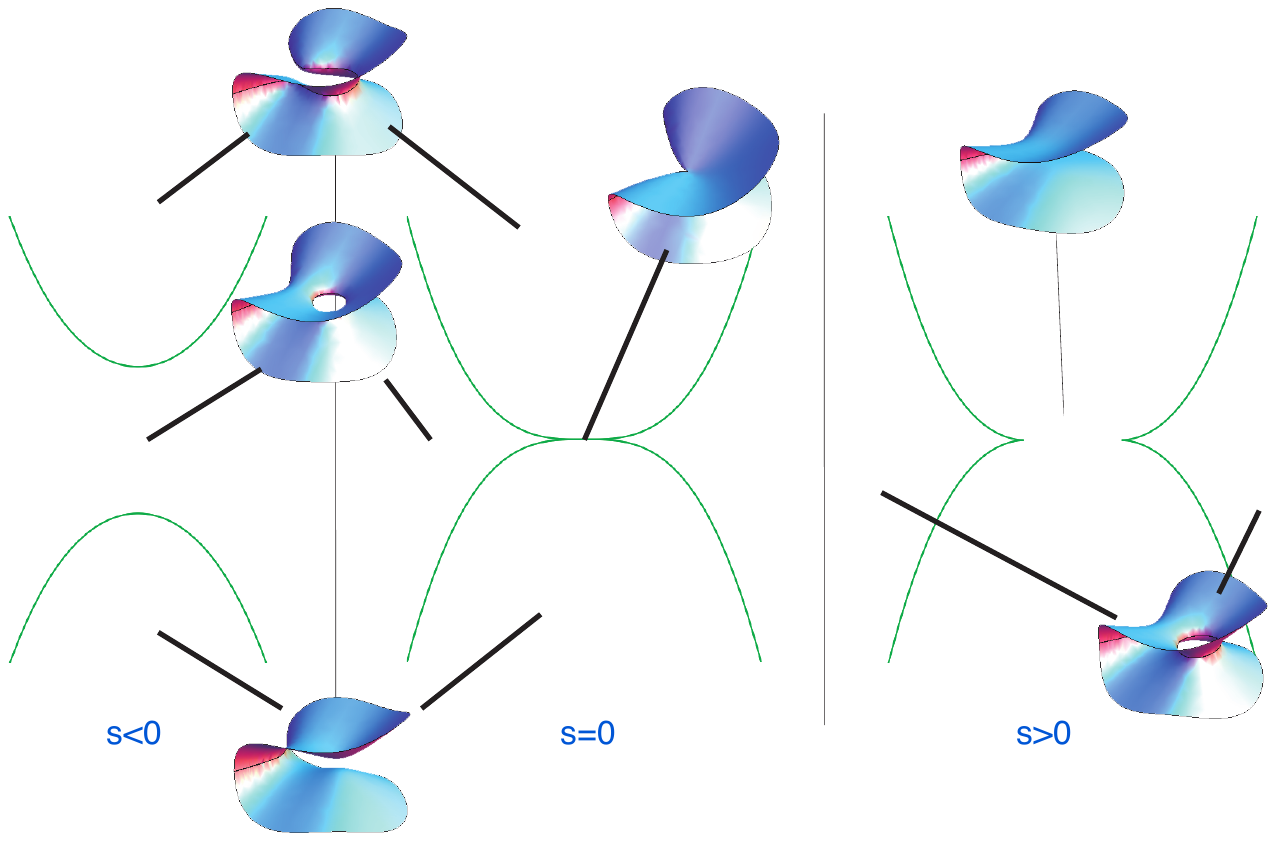}
    \caption{{\bf Merging Move:}   $m_s(x, y, z, t)=(t, x^3-3x(s-t^2)+y^2-z^2)$. For $s<0$ the critical set shows two connected components. Then when $s=0$ these two components touch, and the fibre above the critical point is shown to develop a singularity. As $s$ becomes postive the critical set (shown in green)  separates again. The black lines indicate the points over which each of the fibres lies.}
    \label{fig:merge}
\end{figure}

\medskip
  \noindent{\bf Move 3 (Flipping, fig. \ref{fig:flip})}
\begin{equation*}
f_s(x, y, z, t)=(t, x^4-x^2s+xt+y^2-z^2)
\end{equation*}
\begin{figure}
    \centering
    \includegraphics[width=1\textwidth]{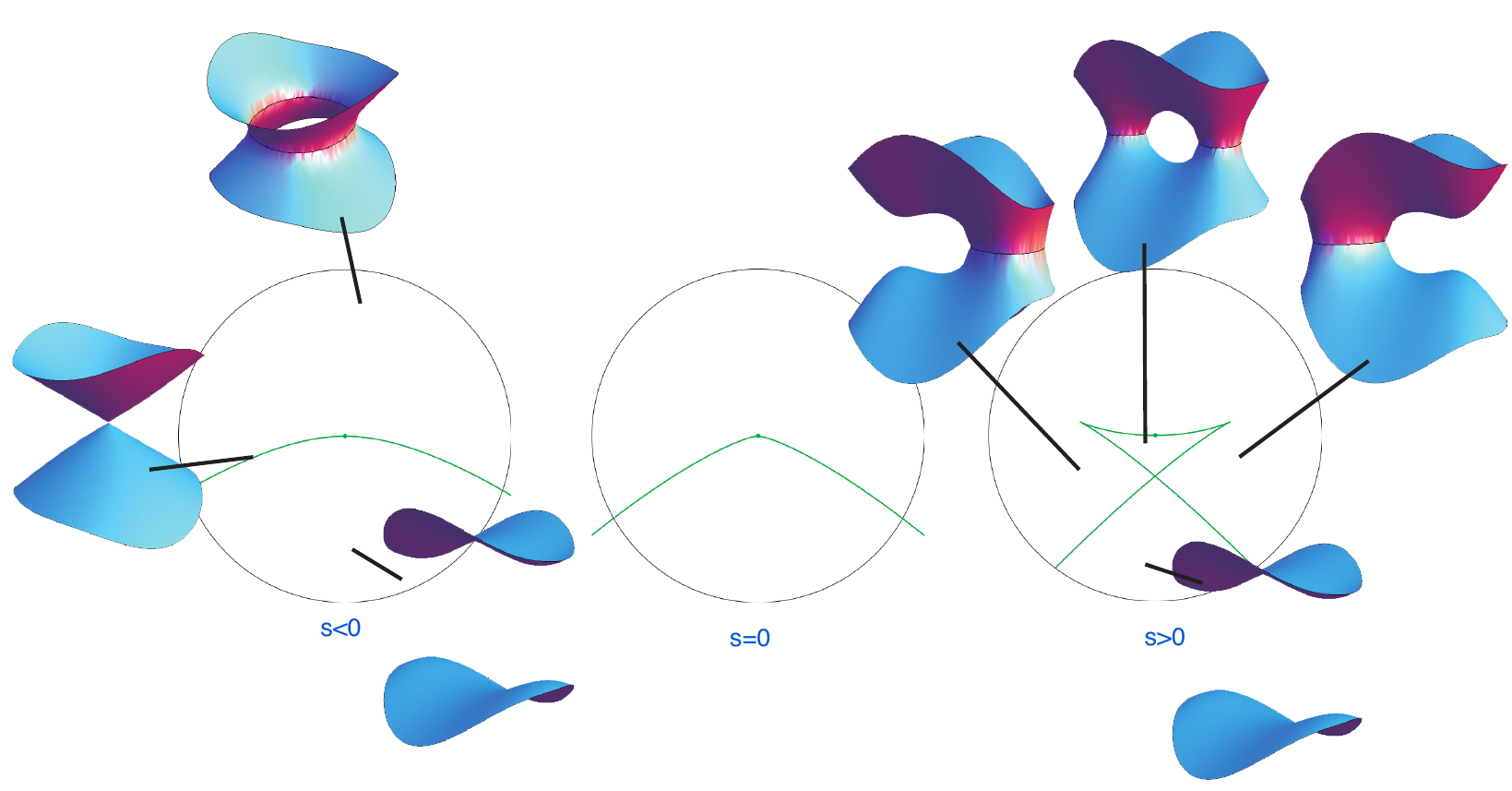}
    \caption{{\bf Flipping Move:}   $f_s(x, y, z, t)=(t, x^4-x^2s+xt+y^2-z^2)$.  For $s<0$ the critical set corresponds to that of a broken Lefschetz fibration. As $s$ becomes postive the critical set (shown in green)  crosses itself. The black lines indicate the points over which each of the fibres lies.}
    \label{fig:flip}
\end{figure}

\medskip
  \noindent{\bf Move 4 (Wrinkling)}
\begin{equation*}
w_s(x, y, z, t)=(t^2-x^2+y^2-z^2+st, 2tx+2yz)
\end{equation*}

The proof of the existence of a Poisson structure presented in \cite{GSV15} also implies that there exists a Poisson structure associated to wrinkled fibrations, where the fibres are leaves of the symplectic foliation and both structures share the same singularities. 

The constructions presented in this paper use a general procedure to construct Poisson structures with prescribed Casimir functions. These ideas, in a general context, can be traced back to Damianou's thesis \cite{Dam89}, to work of Damianou-Petalidou \cite{DamP12}, and have also been attributed to Flaschka-Ratiu. However, the synthetic nature of the proof in \cite{GSV15}, while complete, does not provide any information about the local forms of the Poisson bivector for wrinkled fibrations. The purpose of this note is to provide these details. In this paper, we continue to exploit the techniques of  \cite{GSV15} in order present local formul{\ae} for both the Poisson bivector and for the symplectic forms on the leaves of Poisson structures whose symplectic foliation and singularites are given by wrinkled fibrations and their deformations. Let us summarise the contributions of this paper in the following:

\begin{theorem}\label{wrinkled-poisson}
Let $X$ be a closed, orientable, smooth $4$--manifold equipped with a wrinkled fibration, or for a fixed $s$ a fibration given by one of the birth, merging, flipping, or wrinkiling moves described above. Then there exists a complete Poisson structure whose symplectic leaves correspond to the fibres of the given fibration structure, and the singularities of both the fibration and the Poisson structures coincide. Moreover:
\begin{enumerate}
\item In a neighbourhood of a cusp singularity the Poisson bivector is given by equation (\ref{biv:cusp}), and the symplectic form on the leaves by equation (\ref{CuspForm}).
\item In a neighbourhood of a birth move the Poisson bivector is given by equation (\ref{biv:birth}), and the symplectic form on the leaves by equation  (\ref{BirthForm}).
\item In a neighbourhood of a merging move the Poisson bivector is given by equation (\ref{biv:merge}), and the symplectic form on the leaves by equation  (\ref{MergeForm}).
\item In a neighbourhood of a flipping move the Poisson bivector is given by equation (\ref{biv:flip}), and the symplectic form on the leaves by equation  (\ref{FlipForm}).
\item In a neighbourhood of a wrinkling move the Poisson bivector is given by equation (\ref{biv:wrinkle}), and the symplectic form on the leaves by equation  (\ref{WrinkleForm}).
\end{enumerate}
\end{theorem}
The existence of a Poisson structure with the stated properties follows from Theorem \ref{T:Const-Poisson}, originally shown in \cite{GSV15}.


Given a Lie algebra $\mathfrak{g}$, the integrability problem for $\mathfrak{g}$ was
solved by Lie's third theorem. It gives the existence of a Lie
group $G$ such that $Lie(G)\cong \mathfrak{g}$. Crainic and
Fernandes found obstructions for the integrability of Lie
algebroids. Their main theorem in \cite{CF03} establishes that any Poisson manifold whose
symplectic leaves have trivial second homotopy groups is integrable.
%
%
%

Therefore we deduce:

\begin{lemma} A Poisson structure associated to a wrinkled fibration structure as in Theorem \ref{wrinkled-poisson}, or to a broken Lefschetz fibration as in \cite{GSV15},  none of whose symplectic leaves are, or contain,  $2$--spheres, is integrable.
\end{lemma}


Definitions and useful related results can be found in section \ref{s:def}. The computations of the bivectors and symplectic forms are carried out in section \ref{S:Symp-Struct-Sing}.  We close with a brief appendix that describes the Mathematica code used to solve the differential equations that appear in the search for the local formul{\ae} of the symplectic forms on the leaves.

{\bf Acknowledgements:} We thank Gerardo Sosa for allowing us to use the figures in this paper, which originally appeared in \cite{S13}, and Luis Garc\'ia-Naranjo for commenting on an earlier version. We thank the referees for suggesting how to simplify the computations, and for pointing out the relationships with integrability and linearization. PSS thanks CIMAT in Guanajuato for its warm hospitality in visits where this work was produced, and CONACyT-M\'exico for supporting various reasearch activities.

\section{Definitions}\label{s:def}

\subsection{Poisson manifolds}\label{ss:Poisson} In this section we include basic facts about Poisson geometry that we will use throughout the paper. Interested readers are invited to consult \cite{V94, DZ05, LGPV13}.
\begin{definition}
\label{D:Poisson}
 A {\em Poisson bracket} (or a {\em Poisson structure}) on a smooth manifold $M$ is a bilinear operation
 $\{\cdot , \cdot \}$ on the set $C^\infty(M)$ of real valued smooth functions on $M$ that satisfies
 \begin{enumerate}
\item[(i)] $( C^\infty(M) , \{\cdot , \cdot \})$ is a Lie algebra.
\item[(ii)]  $\{gh, k\}=g\{h, k\}+ h\{g, k\}$ for any $g,h,k\in  C^\infty(M)$.
\end{enumerate}
\end{definition}
A manifold $M$ with such a Poisson bracket is called a {\em Poisson manifold}. Symplectic manifolds $(M,\omega)$ provide examples of Poisson manifolds. In the symplectic case the bracket of $M$ is defined by
\begin{equation*}
\{g,h\}=\omega(X_g,X_h).
\end{equation*}
Hamiltonian vector fields $X_h$ in symplect maniffolds are defined by ${\bf i}_{X_h}\omega =dh$.

Property (ii) in Definition \ref{D:Poisson} allows us to define Hamiltonian vector fields for Poisson manifolds.  For $h\in C^\infty(M)$  we assign it
the {\em Hamiltonian vector field} $X_h$, defined via
\begin{equation*}
X_h(\cdot )=\{\cdot , h \}.
\end{equation*}

It follows from  (ii) that a Poisson bracket $\{g,h\}$ depends solely on the first derivatives of $g$ and $h$.
Hence we may think of the bracket as defining a bivector field $\pi$ defined by
\begin{equation}\label{eq:bracket-bivector}
\{g,h\}=\pi(dg,dh).
\end{equation}

A Poisson bivector $\pi$ can be described locally, for coordinates $(x^1, \dots , x^n)$, by
\begin{equation*}
\pi(x)=\frac{1}{2}\sum_{i,j=1}^n\pi^{ij}(x)\frac{\partial}{\partial x^i}\wedge \frac{\partial}{\partial x^j}.
\end{equation*}
Here $\pi^{ij}(x)=\{x^i,x^j\}=-\{x^j,x^i\}$.
\newline

Poisson brackets satisfy the Jacobi identity, which translates
into a P.D.E. for the components of the Poisson bivector
\cite{LGPV13}. In this paper we will reformulate it for the class
of bivectors we are interested in, so that it can be verified
computationally (see Proposition \ref{eqn:equivJacobi}).





Given a bivector  $\pi$ on $M$, a point $q\in M$, and  $\alpha_q\in T_q^*M$  it is possible to to define:
$$\mathcal{B}:T^*M\to TM\,;\, \mathcal{B}_q(\alpha_q)(\cdot)=\pi_q(\cdot ,\alpha_q) $$

When $\pi$ is Poisson, we have that $X_h=\mathcal{B}(dh)$.

We then define the {\em rank} of  $\pi$ at  $q\in M$ to be equal to the rank of $\mathcal{B}_q:T^*_qM\to T_qM$. This is also the rank of the matrix $\pi^{ij}(x)$.

The distribution defined by $\mathcal{B}_q$  on $ T_qM$ is called the {\em characteristic distribution} of $\pi$. What is known as the  {\em Symplectic Stratification Theorem} is the celebrated statement that this characteristic distribution of a Poisson tensor $\pi$ gives rise to a (possibly singular) foliation by symplectic leaves. This foliation is integrable in the sense of Stefan-Sussman \cite{DZ05}.

Call  $\Sigma_q$  the symplectic leaf of $M$ through the point $q$. As a set $\Sigma_q$ is also the collection of points that may be joined via piecewise smooth integral curves of Hamiltonian vector fields. Write $\omega_{\Sigma_q}$ for the symplectic form on  $\Sigma_q$. Observe that $T_q\Sigma_q$ is exactly the characteristic distribution of $\pi$ through $p$. Therefore, given $u_q, v_q \in T_q\Sigma_q$
there exist $\alpha_q, \beta_q\in T^*_qM$ that under $\mathcal{B}_q$ go to $u_q$ and $v_q$. Using this we can describe $\omega_{\Sigma_q}$:
\begin{equation}
\label{E:Symp-form-gen}
\omega_{\Sigma_q}(q)(u_q, v_q)=\pi_q(\alpha_q, \beta_q)=\langle \alpha_q,  v_q \rangle=-\langle \beta_q,  u_q \rangle.
\end{equation}

As the rank varies, so do the dimensions of the symplectic leaves of the foliation.

In the special case that the rank of the characteristic distribution of a bivector is less than or equal to two, the following were shown to hold in \cite{GSV15}:
\begin{proposition}
\label{P:rank2}

\begin{enumerate}
\item If $\pi$ is a bivector field on $M$ whose characteristic distribution is integrable and has rank less than or equal to two
at each point, then $\pi$ is Poisson.
\item Let $\pi$ be a Poisson structure on $M$ whose rank at each point is less than or equal to two. Then $\pi_1:=k\pi$ is also a Poisson
structure where $k\in C^\infty(M)$ is an arbitrary non-vanishing function.
\end{enumerate}
\end{proposition}

In order to describe the bivectors locally we will use certain Casimir functions.

\begin{definition}
\label{D:Casimir}
Let $M$ be a Poisson manifold. A function $h\in C^\infty(M)$ is called a {\em Casimir} if $\{h,g\}=0$ for
every $g\in C^\infty(M)$.

Equivalently $\mathcal{B}(dh)=0$.
\end{definition}

 The next result was shown in \cite{GSV15}:.
\begin{theorem}
\label{T:Const-Poisson}
Let $M$ be an orientable $n$-manifold, $N$ an orientable $n-2$ manifold, and $f:M\to N$ a smooth map.
Let $\mu$ and $\Omega$ be orientations of $M$ and $N$ respectively. The bracket on $M$ defined by
\begin{equation}
\label{E:Def-Intrinsic}
\{g,h\}\mu=k\,dg\wedge dh \wedge f^*\Omega
\end{equation}
where $k$ is any non-vanishing function on $M$ is Poisson. Moreover, its symplectic leaves are
\begin{enumerate}
\item[(i)] the 2-dimensional leaves $f^{-1}(s)$ where $s\in N$ is
a regular value of $f$, \item[(ii)] the 2-dimensional leaves
$f^{-1}(s)\setminus \{\mbox{Critical Points of $f$}\}$ where $s\in
N$ is a singular value of $f$. \item[(iii)] the 0-dimensional
leaves corresponding to each critical point.
\end{enumerate}
\end{theorem}

Formula \eqref{E:Def-Intrinsic} appeared in \cite{Dam89}
(attributed to H. Flaschka and T. Ratiu).

\begin{definition}
A Poisson manifold $M$ is said to be complete if every Hamiltonian
vector field on $M$ is complete.
\end{definition}

Notice that $M$ is complete if and only if every symplectic leaf is bounded in the
sense that its closure is compact.

\section{Local expressions for the Poisson structures.}\label{S:Symp-Struct-Sing}

\subsection{Local foruml{\ae} for the Poisson bivectors.}

We will now construct explicit expressions for the Poisson structure and the corresponding symplectic forms in a neighbourhood of cusp singularities of wrinkled fibrations $ X \to \Sigma$, as well as for all the possible moves described above. All of the expressions that we will give depend upon a choice of a non-vanishing function $k\in C^\infty(X)$ (see \cite{GSV15}).

Before proceeding we will describe the general strategy employed to find the local bivectors.
\medskip

{\bf Step 1:} Consider the coordinate functions $C_1, C_2$ that describe each fibration as Casimir functions for the Poisson structure that we want to find.
\medskip

{\bf Step 2:} Calculate the differentials $dC_1, dC_2$ of the Casimirs $C_1, C_2$.
\medskip

{\bf Step 3:} We use formula \ref{E:Def-Intrinsic} to compute the
skew-symmetric matrix with entries:
\begin{equation*}
\pi^{ij}=\{x^i,x^j\}\mu=\,dx^i\wedge dx^j \wedge dC_1\wedge dC_2.
\end{equation*}

This matrix will then annihilate $dC_1, dC_2$, as this matrix is to be the endomorphism $\mathcal{B}$ associated to a Poisson structure with $dC_1$ and $dC_2$ as Casimirs.
The components of the bivector field will be given by:
\begin{equation*}
\{x^i, x^j\}=\det\left( \epsilon^i, \epsilon^j, dC_1, dC_2 \right)
\end{equation*}
Here $\epsilon^i$ is the $4\times1$ canonical basis column vector, whose $i$-th
component is $1$ and all others are zero.

For the cusp singularity and the birth, merging, and flipping moves, $t$ is a Casimir, so we
only have to compute
$$\{x_,y\}, \quad  \{x, z\}  \quad \mbox{and} \quad  \{y, z\}.$$
 In fact, for these four cases if we denote
by $dC_2^i$ the components of the column vector $dC_2$ we obtain
\begin{equation*}
\begin{array}{cc}
\{x, y\}= & -dC_2^3 \\
\{x, z\}= & dC_2^2 \\
\{y, z\}= & -dC_2^1
\end{array}
\end{equation*}

\medskip

{\bf Step 4:} According to Proposition \ref{P:rank2} (ii), we
write the Poisson bivector using the skew-symmetric matrix
entries.

Near a wrinkling move the Poisson bivector will be obtained
using formula \ref{E:Def-Intrinsic}. For the other cases the bivector admits a
general expression given by:

\begin{equation}\label{biv:grl}\pi = k(x,y,z,t)\left[ -dC_2^3\frac{\partial}{\partial x}\wedge\frac{\partial}{\partial y}+dC_2^2 \frac{\partial}{\partial x}\wedge \frac{\partial}{\partial z} -dC_2^1\frac{\partial}{\partial y}\wedge \frac{\partial}{\partial z} \right] \end{equation}

for any non-vanishing smooth function $k$.

\medskip

The corresponding results for neighborhoods of Lefschetz singularities and broken singular circles were obtained in \cite{GSV15}. We proceed to describe the results this general strategy yields for cusp singularities and the moves described above.

Attentive readers might wonder why we are not using equation \ref{E:Def-Intrinsic} directly. The computations needed to be carried out to take the volume form from the left hand side to the right hand side of the equation  \ref{E:Def-Intrinsic} are cumbersome and lengthy. The method we present below yields equivalent results, and may be easily verified.
\subsubsection{Local expressions near a cusp singularity.}\label{SS:SingularCusp}

The local coordinate model around a cusp singularity is given by:
\begin{equation*} (x, y, z, t) \mapsto (C_1(x, y, z, t), C_2(x, y, z, t))=(t, x^3-3xt+y^2-z^2) \end{equation*}

The differentials $dC_1$ and $dC_2$ are therefore:
\begin{equation*}
\begin{array}{ccccc}
 dC_1 = & (0 & 0 & 0 & 1 )\\
 dC_2 = & (3x^2-3t & 2y & -2z & -3x ) \\
\end{array}
\end{equation*}

The following matrix annihilates $dC_1$ and $dC_2$ and
its entries satisfy the Jacobi identity :

\begin{equation*}\left(
\begin{array}{cccc}
 0 & 2 k z & 2 k y & 0 \\
 -2 k z & 0 & k(3 t-3 x^2) & 0 \\
 -2 k y & k(3 x^2-3 t) & 0 & 0 \\
 0 & 0 & 0 & 0
\end{array}
\right)\end{equation*}

Which means that the Poisson bivector in the local coordinates of a cusp singularity is described by:
\begin{equation}\label{biv:cusp}\pi = k(x,y,z,t)\left[ 2z\frac{\partial}{\partial x}\wedge\frac{\partial}{\partial y} + 2y \frac{\partial}{\partial x}\wedge \frac{\partial}{\partial z} +(3 t-3 x^2)\frac{\partial}{\partial y}\wedge \frac{\partial}{\partial z} \right] \end{equation}

\subsubsection{Local expressions near a birth move.}\label{SS:MoveBirth}

The local coordinate model around a birth move is given by:
\begin{equation*}b_s(x, y, z, t)=(C_1(x, y, z, t), C_2(x, y, z, t))=(t, x^3-3x(t^2-s)+y^2-z^2)\end{equation*}
The differentials $dC_1$ and $dC_2$ are therefore:
\begin{equation*}
\begin{array}{ccccc}
  dC_1 = & (0 & 0 & 0 & 1) \\
  dC_2 = & (3x^2-3(t^2-s) & 2y & -2z &  -6xt)\\
\end{array}
\end{equation*}

From which we can obtain the following matrix:
\begin{equation*}\left(
\begin{array}{cccc}
 0 & 2 k z & 2 k y & 0 \\
 -2 k z & 0 & k \left(3 \left(t^2-s\right)-3 x^2\right) & 0 \\
 -2 k y & k \left(3 x^2-3 \left(t^2-s\right)\right) & 0 & 0 \\
 0 & 0 & 0 & 0
\end{array}
\right)\end{equation*}

Hence the Poisson bivector near a birth move has the form:
\begin{equation}\label{biv:birth}\pi_s = k(x,y,z,t)\left[ 2z\frac{\partial}{\partial x}\wedge\frac{\partial}{\partial y} + 2y \frac{\partial}{\partial x}\wedge \frac{\partial}{\partial z} - 3(s-t^2+x^2)\frac{\partial}{\partial y}\wedge \frac{\partial}{\partial z} \right] \end{equation}

\subsubsection{Local expressions near a merging move.}\label{SS:MoveMerging}
The local coordinate model around a merging move is given by:
\begin{equation*}m_s(x, y, z, t)=(C_1(x, y, z, t), C_2(x, y, z, t))=(t, x^3-3x(s-t^2)+y^2-z^2)\end{equation*}

The differentials $dC_1$ and $dC_2$ are therefore:
\begin{equation*}
\begin{array}{ccccc}
 dC_1 = & ( 0 & 0 & 0 & 1 )\\
  dC_2 = & ( 3x^2-3(s-t^2) & 2y & -2z & 6xt) \\
\end{array}
\end{equation*}

The associated matrix is then:
\begin{equation*}\left(
\begin{array}{cccc}
 0 & 2 k z & 2 k y & 0 \\
 -2 k z & 0 & k \left(3 \left(s-t^2\right)-3 x^2\right) & 0 \\
 -2 k y & k \left(3 x^2-3 \left(s-t^2\right)\right) & 0 & 0 \\
 0 & 0 & 0 & 0
\end{array}
\right)\end{equation*}

So the Poisson bivector in a neighbourhood of a merging move is described as:
\begin{equation}\label{biv:merge}\pi_s = k(x,y,z,t)\left[ 2z\frac{\partial}{\partial x}\wedge\frac{\partial}{\partial y} + 2y\frac{\partial}{\partial x}\wedge \frac{\partial}{\partial z} - 3(s-t^2-x^2)\frac{\partial}{\partial y}\wedge \frac{\partial}{\partial z} \right] \end{equation}

\subsubsection{Local expressions near a flipping move.}\label{SS:MoveFlipping}
The local coordinate model around a flipping move is given by:
\begin{equation*}f_s(x, y, z, t)=(C_1(x, y, z, t), C_2(x, y, z, t))=(t, x^4-x^2s+xt+y^2-z^2)\end{equation*}

The differentials $dC_1$ and $dC_2$ are therefore:
\begin{equation*}
\begin{array}{ccccc}
 dC_1 = & (  0 & 0 & 0 & 1 ) \\
 dC_2 = & ( 4x^3-2xs+t & 2y & -2z & x  ) \\
\end{array}
\end{equation*}

The corresponding matrix is:
\begin{equation*}\left(
\begin{array}{cccc}
 0 & 2 k z & 2 k y & 0 \\
 -2 k z & 0 & k \left(-4 x^3+2 s x-t\right) & 0 \\
 -2 k y & k \left(4 x^3-2 s x+t\right) & 0 & 0 \\
 0 & 0 & 0 & 0
\end{array}
\right)\end{equation*}

The Poisson bivector in a neighborhood of a flipping move can then
be written in the following way:
\begin{equation}\label{biv:flip}\pi_s = k(x,y,z,t)\left[2z\frac{\partial}{\partial x}\wedge\frac{\partial}{\partial y} + 2y\frac{\partial}{\partial x}\wedge \frac{\partial}{\partial z} - (t-2sx+4x^3)\frac{\partial}{\partial y}\wedge \frac{\partial}{\partial z} \right] \end{equation}

\subsubsection{Local expressions near a wrinkling move.}\label{SS:MoveWrinkling}
The local coordinate model around a wrinkling move is given by:
\begin{equation*}w_s(x, y, z, t)=(C_1(x, y, z, t), C_2(x, y, z, t))=(t^2-x^2+y^2-z^2+st, 2tx+2yz)\end{equation*}
The differentials $dC_1$ and $dC_2$ are therefore:
\begin{equation*}
\begin{array}{ccccc}
 dC_1 = & (  -2x & 2y & -2z & 2t+s  )\\
 dC_2 = & (2t & 2z & 2y & 2x ) \\
\end{array}
\end{equation*}

The matrix we are interested in is given by:
\begin{equation*}\left(
\begin{array}{cccc}
 0 & k (s y+2 t y+2 x z) & k (2 x y-(s+2 t) z) & -2 k \left(y^2+z^2\right) \\
 -k (s y+2 t y+2 x z) & 0 & k \left(s t+2 \left(t^2+x^2\right)\right) & k (2 t z-2 x y) \\
 k ((s+2 t) z-2 x y) & -k \left(s t+2 \left(t^2+x^2\right)\right) & 0 & 2 k (t y+x z) \\
 2 k \left(y^2+z^2\right) & 2 k (x y-t z) & -2 k (t y+x z) & 0
\end{array}
\right)\end{equation*}

The expression for the Poisson bivector in a neighbourhood of a wrinkling move is then:

\begin{equation}\label{biv:wrinkle}\pi_s  =  k(x,y,z,t) [ (-2 s y - 4 t y - 4 x z)\frac{\partial}{\partial x}\wedge\frac{\partial}{\partial y} + (-4 x y + 2 s z + 4 t z)\frac{\partial}{\partial x}\wedge \frac{\partial}{\partial z} +(4 y^2 + 4 z^2)\frac{\partial}{\partial x}\wedge \frac{\partial}{\partial t} + \end{equation}

\begin{equation*} \quad \,\, - (2 s t + 4 t^2 + 4 x^2) \frac{\partial}{\partial y}\wedge \frac{\partial}{\partial z} + 4(x y - t z) \frac{\partial}{\partial y}\wedge \frac{\partial}{\partial t} -4(t y + x z)\frac{\partial}{\partial z}\wedge \frac{\partial}{\partial t}  ]\end{equation*}

\subsubsection{Linearization}
We follow chapters 3 and 4 of \cite{DZ05}, where more details and examples may be found.
Let $\mathfrak{l}$ be a finite-dimensional Lie algebra. Denote by
$\mathfrak{r}$ the radical of $\mathfrak{l}$, i.e., the maximal
solvable ideal of $\mathfrak{l}$. Then $\mathfrak{g}=\mathfrak{l}/
\mathfrak{r}$ is a semi-simple Lie algebra. The Levi-Malcev theorem
states that $\mathfrak{l}$ can be decomposed as a semi-direct
product:
\begin{equation*}
\mathfrak{l}=\mathfrak{g}\ltimes\mathfrak{r}.
\end{equation*}
In analogy with this Levi-Malcev decomposition we have a Levi
decomposition for Poisson structures. Let $\pi$ be a Poisson
structure and denote by $\pi_0$ its linear part. By definition
we obtain that $\pi_0$ generates a Lie algebra $\mathfrak{l}$. We
take the Levi-Malcev decomposition of $\mathfrak{l}$, with the previous
notation. Let $\{x_1, ..., x_m, y_1, ..., y_m\}$ be a basis for
$\mathfrak{l}$, such that $\{x_1, ..., x_m\}$ span $\mathfrak{g}$,
and $\{y_1, ..., y_m\}$ spans a complement $\mathfrak{r}$ of
$\mathfrak{g}$ with respect to the adjoint action of $\mathfrak{g}$
on $\mathfrak{l}$.

A Levi decomposition for $\pi$ at $0$, with $\pi(0)=0$, provides local coordinates such that;
\begin{equation*}\{x_i, x_j\}\in\mathfrak{g} \quad \mbox{and}\quad \{x_i, y_j\}\in\mathfrak{r}.
\end{equation*}

Any analytic Poisson structure $\pi$,
which vanishes at $0$, admits a Levi decomposition in a neighborhood of $0$.

Now we will focus on the expressions for the bivectors obtained
in equations (\ref{biv:cusp}), (\ref{biv:birth}), (\ref{biv:merge}), (\ref{biv:flip}), and (\ref{biv:wrinkle}). We fix $k\equiv1$. It can be seen that in the case
of cusp singularities, the linear part of the corresponding
Poisson structure (\ref{biv:cusp}) defines a Lie algebra through the commutation
relations:
\begin{equation*}
[x, z]=2y\quad \quad [x, y]=2y \quad \quad [y, z]=3t
\end{equation*}

For the Birth, Merge and Flipping moves, corresponding to the bivectors (\ref{biv:birth}), (\ref{biv:merge}), and (\ref{biv:flip}), respectively, their linear part in all these cases is generated by:
\begin{equation*}
[x, z]=2y\quad \quad [x, y]=2y
\end{equation*}

Notice that this Lie algebra contains a nonzero
Abelian ideal, hence it is not
semi-simple.

So in all these cases the linear part of the Poisson structure admits a
decomposition of the form $\mathbb{R}\times L_3$, where $L_3$ is a
semi-direct product of Lie algebras:
\begin{equation*}
L_3=\mathbb{R}\ltimes_A\mathbb{R}^2
\end{equation*}

Here $\mathbb{R}$ acts on $\mathbb{R}^2$ linearly by the matrix:
\begin{equation*}A=\left(
\begin{array}{cccc}
 2 & 0 \\
 0 & 2
\end{array}
\right)\end{equation*}

For the case of wrinkled fibrations, corresponding to (\ref{biv:wrinkle}), we see that all
the commutation relations are trivial. Therefore the corresponding
linear part of its Poisson structure spans an Abelian Lie algebra,
which is not semi-simple.

Conn's theorem asserts that, provided the linear part of an
analytic Poisson structure $\pi$ that vanishes at $0$, corresponds
to a semi-simple Lie algebra, $\pi$ admits a local analytic
linearizaton at $0$. Hence, in the spirit of Conn's theorem, the
linearization of all these Poisson structures is not guaranteed. Moreover, the semi-direct product $L_3$ is degenerate (formally, analytically, and smoothly).

Finally, in the general case when $k$ is a non vanishing smooth function, we
obtain other Poisson structures.

\begin{question}Does
there exist a function $k$ such that an expression given by one of the bivectors  (\ref{biv:cusp}), (\ref{biv:birth}), (\ref{biv:merge}), (\ref{biv:flip}) or (\ref{biv:wrinkle})  is linearizable?
\end{question}

\subsection{Equations for the symplectic forms on the leaves near singularities.}

\begin{proposition}
Let $q\in B^4\backslash \{0\}$ and let $\pi$ be the Poisson structure
near a cusp singularity. Then the symplectic form induced by $\pi$
on the symplectic leaf $\Sigma_q$ through $q=(x, y, z, t)$ at the
point $q$ is given by
\begin{equation}\label{CuspForm}
\omega_{\Sigma_q}=\frac{1}{3 \left(x^2-t\right)}\omega_{Area}(q)
\end{equation}

here $\omega_{Area}$ is the area form on $\Sigma_q$ induced by
the euclidean metric on $B^4$.

\end{proposition}

\begin{proof}

If $u_q, v_q$ are tangent vectors to the leaves there exist
co-vectors $\alpha_q, \beta_q\in T^*_qM$ such that $\mathcal{B}_q
(\alpha_q)=u_q$ and $\mathcal{B}_q (\beta_q)=v_q$, where
$\mathcal{B}_q$ is given by the rule:
\begin{equation*}
\mathcal{B}_q(\alpha)(\cdot)=\pi_q(\cdot, \alpha)
\end{equation*}

Finding two tangent vectors to the symplectic leaves is equivalent
to detecting vectors annihilated simultaneously by the
differential of two Casimir functions for the corresponding
Poisson structure. Note that the characteristic distribution has
rank 2.

In this case we find that the vectors:
\begin{equation*}
u_q=-\frac{1}{3 \left(t-x^2\right)}(2 z\frac{\partial}{\partial
x}+ 3 \left(t-x^2\right)\frac{\partial}{\partial z})
\end{equation*}

\begin{equation*}
v_q=\frac{1}{3 (t - x^2)}(2y\frac{\partial}{\partial x}+3 (t -
x^2)\frac{\partial}{\partial y})
\end{equation*}

are tangent to $\Sigma_q$ at $q$. Using the local expression of
the Poisson structure for a cusp singularity given by equation
(\ref{biv:cusp}), one can check that
$\mathcal{B}_q(\alpha_q)=u_q$, for
\begin{equation*}
\alpha_q=-\frac{dy}{k(x, y, z, t)3 (t - x^2)}.
\end{equation*}

Similarly, $\mathcal{B}_q(\beta_q)=v_q$, for
\begin{equation*}
\beta_q= \frac{1}{k(x, y, z,t)}(-\frac{1}{2 z} dx+\frac{y}{3 (t -
x^2) z} dy).
\end{equation*}

A direct calculation now implies that the symplectic form is:
\begin{equation*}
\omega_{\Sigma_q}(q)(u_q, v_q)=\langle \alpha_q,  v_q
\rangle=\frac{1}{3 (x^2-t)}\omega_{Area}(q)
\end{equation*}

\end{proof}

\begin{proposition}
Let $q\in B^4\backslash \{0\}$ and $\pi_s$ be the Poisson structure
near a birth move. The symplectic form induced by $\pi_s$ on the
symplectic leaf $\Sigma_q$ through $q=(x, y, z, t)$ at the point
$q$ is given by
\begin{equation}\label{BirthForm}
\omega_{\Sigma_q}=\frac{1}{k(x, y, z, t)(3 (s - t^2 +
x^2)}\omega_{Area}(q)
\end{equation}

here $\omega_{Area}$ is the area form on $\Sigma_q$ induced by
the euclidean metric on $B^4$.

\end{proposition}

\begin{proof}

Making use of the corresponding Casimir functions for the Poisson
structure associated to a birth move we obtain that the vectors
\begin{equation*}
u_q=\frac{1}{3 \left(s-t^2+x^2\right)}(2 z\frac{\partial}{\partial
x}+ 3 \left(s-t^2+x^2\right)\frac{\partial}{\partial z})
\end{equation*}

\begin{equation*}
v_q=-\frac{1}{3 (s - t^2 + x^2)}(2y\frac{\partial}{\partial x}+3
(s - t^2 + x^2)\frac{\partial}{\partial y})
\end{equation*}

are tangent to $\Sigma_q$ at $q$. Using the local expression
(\ref{biv:birth}) of the Poisson structure one can check that
$\mathcal{B}_q (\alpha_q)=u_q$, for
\begin{equation*}
\alpha_q=\frac{dy}{k(x, y, z, t)3 (s - t^2 + x^2)}.
\end{equation*}

Similarly, $\mathcal{B}_q(\beta_q)=v_q$, for
\begin{equation*}
\beta_q=\frac{1}{k(x,y, z, t)}(-\frac{1}{2 z} dx-\frac{y}{3 (s -
t^2 + x^2) z} dy)\end{equation*}

The expression for the symplectic form follows from:
\begin{equation*}
\omega_{\Sigma_q}(q)(u_q, v_q)=\langle \alpha_q,  v_q
\rangle=-\langle \beta_q,  u_q \rangle
\end{equation*}\end{proof}

\begin{proposition}
Let $q\in B^4\backslash \{0\}$ and let $\pi_s$ be the Poisson structure
near a merging move. The symplectic form induced by $\pi_s$ on the
symplectic leaf $\Sigma_q$ through $q=(x, y, z, t)$ at the point
$q$ is given by
\begin{equation}\label{MergeForm}
\omega_{\Sigma_q}=\frac{1}{3 (t^2-s + x^2)}\omega_{Area}(q)
\end{equation}
here $\omega_{Area}$ is the area
form on $\Sigma_q$ induced by the euclidean metric on $B^4$.

\end{proposition}

\begin{proof}
We proceed similarly to the previous cases above. We find that the
vectors
\begin{equation*}
u_q=\frac{1}{3 \left(s-t^2-x^2\right)}(-2
z\frac{\partial}{\partial x}+
3\left(s-t^2-x^2\right)\frac{\partial}{\partial z})
\end{equation*}

\begin{equation*}
v_q=\frac{1}{3 (s - t^2 - x^2)}(2y\frac{\partial}{\partial x}+3 (s
- t^2 - x^2)\frac{\partial}{\partial y})
\end{equation*}
are tangent to $\Sigma_q$ at $q$. Using the local expression
(\ref{biv:merge}) of the corresponding Poisson structure one can
check that $\mathcal{B}_q (\alpha_q)=u_q$, for
\begin{equation*}
\alpha_q=-\frac{dy}{k(x, y, z, t)3 (s - t^2 - x^2)}.
\end{equation*}

Similarly, $\mathcal{B}_q(\beta_q)=v_q$, for
\begin{equation*}
\beta_q= \frac{1}{{k(x, y, z, t)}}(-\frac{1}{2 z} dx+\frac{y}{3 (s
- t^2 - x^2) z} dy).
\end{equation*}

As before the symplectic form is obtained by computing:
\begin{equation*}
\omega_{\Sigma_q}(q)(u_q, v_q)=\langle \alpha_q,  v_q
\rangle=-\langle \beta_q,  u_q \rangle
\end{equation*}

\end{proof}

\begin{proposition}
Let $q\in B^4\backslash \{0\}$ and let $\pi_s$ be the Poisson structure
near a flipping move. The symplectic form induced by $\pi_s$ on the
symplectic leaf $\Sigma_q$ through $q=(x, y, z, t)$ at the point
$q$ is given by
\begin{equation}\label{FlipForm}
\omega_{\Sigma_q}=\frac{1}{k(x, y, z, t)(t - 2 s x + 4
x^3)}\omega_{Area}(q)
\end{equation}
here $\omega_{Area}$ is the area form on $\Sigma_q$ induced by
the euclidean metric on $B^4$.

\end{proposition}

\begin{proof}

In this case, the following vectors:
\begin{equation*}
u_q=\frac{1}{t - 2 s x + 4 x^3}(2 z\frac{\partial}{\partial x}+ t
- 2 s x + 4 x^3\frac{\partial}{\partial z})
\end{equation*}

\begin{equation*}
v_q=\frac{1}{t - 2 s x + 4 x^3}(-2y\frac{\partial}{\partial x}+t -
2 s x + 4 x^3\frac{\partial}{\partial y})
\end{equation*}

are tangent to $\Sigma_q$ at $q$. Using equation (\ref{biv:flip})
we can check that $\mathcal{B}_q (\alpha_q)=u_q$, for
\begin{equation*}
\alpha_q=-\frac{dy}{k(x, y, z, t)(-t + 2 s x - 4 x^3)}.
\end{equation*}

Similarly, $\mathcal{B}_q(\beta_q)=v_q$, for
\begin{equation*}
\beta_q= \frac{1}{k(x, y, z, t)}(-\frac{1}{2 z} dx-\frac{y}{(t - 2
s x + 4 x^3) z} dy).
\end{equation*}

A straightforward calculation shows that:
\begin{equation*}
\omega_{\Sigma_q}=\frac{1}{k(x, y, z, t)(t - 2 s x + 4 x^3)
z}\omega_{Area}(q)
\end{equation*}

\end{proof}

\begin{proposition}
Let $q\in B^4\backslash \{0\}$ and let $\pi_s$ be the Poisson structure
near a wrinkling move. The symplectic form induced by $\pi_s$ on the
symplectic leaf $\Sigma_q$ through $q=(x, y, z, t)$ at the point
$q$ is given by
\begin{equation}\label{WrinkleForm}
\omega_{\Sigma_q}=-\frac{1}{2 (t y + x z)}\omega_{Area}(q)
\end{equation}
here $\omega_{Area}$ is the area form on $\Sigma_q$ induced by the
euclidean metric on $B^4$.
\end{proposition}

\begin{proof}
Using the corresponding Poisson structure for a wrinkling move move given in equation (\ref{biv:wrinkle}) we obtain:
\begin{equation*}
u_q=-\frac{1}{2 (t y + x z)}((2 x y - s z - 2 t
z)\frac{\partial}{\partial x}+(s t + 2 t^2 + 2
x^2)\frac{\partial}{\partial y}-2 (t y + x
z)\frac{\partial}{\partial t})
\end{equation*}

\begin{equation*}
v_q=-\frac{1}{(t y + x z)}((y^2 + z^2)\frac{\partial}{\partial
x}+(x y - t z)\frac{\partial}{\partial y}-(t y + x
z)\frac{\partial}{\partial z})
\end{equation*}

and makes $v_q$ an unitary vector. These vectors are tangent to
$\Sigma_q$ at $q$. Using the local expression of the Poisson
structure in  (\ref{biv:wrinkle}) we check that $\mathcal{B}_q
(\alpha_q)=u_q$, for
\begin{equation*}
\alpha_q=\frac{1}{k(x, y, z, t)}(\frac{1}{2 (y^2 +
z^2)}dx-\frac{-2 x y + s z + 2 t z}{4 (t y + x z) (y^2 + z^2)}dt)
\end{equation*}

Similarly, $\mathcal{B}_q(\beta_q)=v_q$, for
\begin{equation*}
\beta_q=\frac{dt}{2 (t y + x z)}
\end{equation*}
The proposition is shown as in the previous cases by calculating the symplectic form explicitly using the above equations.
\end{proof}

\appendix

{\small

\medskip

}


\end{document}